\documentclass[reqno,10pt]{amsart}

\usepackage{graphicx}
\usepackage{amsmath}
\usepackage{amssymb}
\usepackage{amsthm}
\usepackage{enumitem}
\usepackage{bbm}
\usepackage[all,error]{onlyamsmath}
\usepackage[strict]{csquotes}
\usepackage{url}
\usepackage{braket,mleftright}

\newcommand{\ie}{\textit{i.e.}\;}
\newcommand{\eg}{\textit{e.g.}\;}
\newcommand{\cf}{\textit{cf.}\;}
\newcommand{\bbC}{\mathbb{C}}
\newcommand{\bbN}{\mathbb{N}}
\newcommand{\bbR}{\mathbb{R}}
\newcommand{\bbZ}{\mathbb{Z}}
\newcommand{\mrm}[1]{\mathrm{#1}}		
\newcommand{\ol}[1]{\overline{#1}}		
\newcommand{\co}{\colon}				
\newcommand{\vrt}{\,\vert\,}			
\newcommand{\lto}{\rightarrow}			
\newcommand{\what}[1]{\widehat{#1}}		
\newcommand{\wtilde}[1]{\widetilde{#1}}	
\newcommand{\setm}{\smallsetminus}		
\newcommand{\mcap}{\cap}				
\newcommand{\mcup}{\cup}				
\newcommand{\op}{\oplus}				
\newcommand{\dsum}{\dotplus}			
\newcommand{\hsum}{\,\what{+}\,}		
\newcommand{\Amin}{A_{\min}}			%
\newcommand{\Amax}{A_{\max}}			%
\newcommand{\Bmin}{B_{\min}}			%
\newcommand{\Bmax}{B_{\max}}			%
\newcommand{\Lmin}{L_{\min}}			%
\newcommand{\Lmax}{L_{\max}}			%
\DeclareMathOperator{\dom}{dom}		
\DeclareMathOperator{\ran}{ran}		
\DeclareMathOperator{\res}{res}		
\DeclareMathOperator{\spn}{span}	
\DeclareMathOperator{\mul}{mul}		
\DeclareMathOperator{\reg}{reg}		

\newcommand{\cG}{\mathcal{G}}
\newcommand{\cH}{\mathcal{H}}

\newcommand{\cS}{\mathcal{S}}

\newcommand{\cX}{\mathcal{X}}

\newcommand{\fH}{\mathfrak{H}}

\newcommand{\fK}{\mathfrak{K}}

\newcommand{\fM}{\mathfrak{M}}
\newcommand{\fN}{\mathfrak{N}}

\newcommand{\wtB}{\wtilde{B}}

\newcommand{\vp}{\varphi}
\newcommand{\wtcG}{\wtilde{\cG}}
\newcommand{\whf}{\what{f}}

\newcommand{\wtk}{\wtilde{k}}
\newtheorem{thm}{Theorem}
\newtheorem{lem}[thm]{Lemma}

\newtheorem{prop}[thm]{Proposition}

\theoremstyle{definition}

\numberwithin{equation}{section}
\numberwithin{thm}{section}
\begin{document}

\title[A Hilbert space model led by the A-model]{The A-model
with mutually equal model parameters can lead to a Hilbert space
model}
\author{Rytis Jur\v{s}\.{e}nas}
\address{Vilnius University,
Institute of Theoretical Physics and Astronomy,
Saul\.{e}tekio ave.~3, LT-10257 Vilnius, Lithuania}
\email{Rytis.Jursenas@tfai.vu.lt}
\date{\today}
\begin{abstract}
It is known that the A-model for higher order singular
perturbations can be considered as a Hilbert space model
if the model parameters are mutually distinct, and that it
is necessarily a Pontryagin space model if otherwise.
In this note we demonstrate that the A-model with mutually
equal model parameters can nonetheless lead to
a Hilbert space model if the extensions
in the model space are instead described by
suitable linear relations.
\end{abstract}
\maketitle
\section{Introduction}
As it is known from \cite{Dijksma05}, the A-model for
rank one perturbations of class $\fH_{-m-2}\setm\fH_{-m-1}$,
$m\in\bbN$, of a lower semibounded self-adjoint operator $L$ in
$\fH_0$ is considered in general from the perspective of an
indefinite inner product space (Pontryagin space), which we
denote by $\cH_{\mrm{A}}$. Here
$(\fH_n,\braket{\cdot,\cdot}_n)_{n\in\bbZ}$
is the scale of Hilbert spaces associated with $L$, and the
$\fH_n$-scalar product is defined via an operator
$b_n(L):=\prod_{j=1}^n(L-z_j)$ with some fixed model parameters
$z_j\in\res L\mcap\bbR$:
$\braket{\cdot,\cdot}_n:=\braket{\cdot,b_n(L)\cdot}_0$.
The rank of indefiniteness of $\cH_{\mrm{A}}$
depends on the Gram matrix $\cG_{\mrm{A}}$ that determines
an indefinite inner product $[\cdot,\cdot]_{\mrm{A}}$ in
$\cH_{\mrm{A}}$.
By definition it is assumed that $\cG_{\mrm{A}}$
is invertible and Hermitian, but for perturbations of
class $\fH_{-4}$ or higher (\ie $m\geq2$),
this is not sufficient in order to apply the extension theory of
operators in $\cH_{\mrm{A}}$.
It appears that for such perturbations additional
restrictions imposed on $\cG_{\mrm{A}}$ are needed;
for example, for mutually equal model parameters
$z_j$, the Gram matrix
$\cG_{\mrm{A}}=([\cG_{\mrm{A}}]_{jj^\prime})$
must be of an anti-triangular form:
\begin{equation}
\begin{split}
[\cG_{\mrm{A}}]_{jj^\prime}=&
[\cG_{\mrm{A}}]_{j^\prime j}\in\bbR\,,\quad
j,j^\prime\in\{1,\ldots,m\}\,,
\\
[\cG_{\mrm{A}}]_{jj^\prime}=&0\,,\quad
j\in\{1,\ldots,m-1\}\,,\quad
j^\prime\in\{1,\ldots,m-j\}\,,
\\
[\cG_{\mrm{A}}]_{jm}=&
[\cG_{\mrm{A}}]_{j+1,m-1}\,,\quad
j\in\{1,\ldots,m-1\}\,.
\end{split}
\label{eq:1}
\end{equation}
More generally (\cite[Theorem~3.2]{Dijksma05}),
if at least two of the $z_j$'s are equal, then
$\cH_{\mrm{A}}$ must have a nontrivial rank of indefiniteness.
In contrast, if the points $z_j$ are all mutually distinct,
then $\cH_{\mrm{A}}$ can be considered as a Hilbert space,
\ie there exists a positive matrix $\cG_{\mrm{A}}$
satisfying all necessary conditions required for the
application of the theory of extensions to $\cH_{\mrm{A}}$
of $L$.

The main goal of this note is to demonstrate
that, for equal $z_j$'s, we still can extract
a Hilbert space model from the A-model provided that
\begin{equation}
[\cG_{\mrm{A}}]_{mm}>0\,,\quad
[\cG_{\mrm{A}}]_{m-1,m}=
[\cG_{\mrm{A}}]_{m,m-1}\in\bbR
\label{eq:2}
\end{equation}
for $m\geq2$.
In fact, we consider rank-$d$ perturbations, with
an arbitrary $d\in\bbN$, so that actually we have that
$\cG_{\mrm{A}}=([\cG_{\mrm{A}}]_{\sigma j,\sigma^\prime j^\prime})$
is a $dm\times dm$ Gram matrix;
the indices $\sigma$, $\sigma^\prime$ range over an
index set $\cS$ of cardinality $d\in\bbN$.
The conditions in \eqref{eq:1}, \eqref{eq:2} are then modified
appropriately
(see \eqref{eq:GAcomm} and \eqref{eq:A2}).

In the A-model, singular perturbations of $L$
in $\cH_{\mrm{A}}$ are specified by the extensions
of a densely defined, closed, symmetric operator
$\Amin$ in $\cH_{\mrm{A}}$, provided an invertible
Hermitian $\cG_{\mrm{A}}$
satisfies appropriate conditions
(for equal $z_j$'s these are as in \eqref{eq:1}).
We recall that
$\Amin$ is the adjoint in $\cH_{\mrm{A}}$ of
the restriction $\Amax\supseteq\Amin$ to $\cH_{\mrm{A}}$ of
the triplet adjoint $\Lmax$ of $\Lmin$. The triplet adjoint
is taken with respect to the
Hilbert triple $\fH_{m}\subseteq\fH_0\subseteq\fH_{-m}$.
The operator $\Lmin$ is densely defined, closed, symmetric
in $\fH_m$, has defect numbers $(d,d)$, and is essentially
self-adjoint in $\fH_0$, whose closure is $L$.
As is usual in extension theory,
an extension $A_\Theta\in\mrm{Ext}(\Amin)$
is parametrized by a linear relation $\Theta$ in
$\bbC^d$ according to
$\dom A_\Theta=\{f\in\dom\Amax\vrt\Gamma f\in\Theta\}$,
where $\Gamma:=(\Gamma_0,\Gamma_1)
\co\dom\Amax\lto\bbC^{d}\times\bbC^d$ defines the
boundary triple $(\bbC^d,\Gamma_0,\Gamma_1)$
for $\Amax=\Amin^*$.

To explain our main idea, let us now consider the
A-model with equal model parameters, $z_j=z_1$. For
simplicity we let $d=1$.
Let $\cH^{\min}_{\mrm{A}}:=\cH_{\mrm{A}}\mcap\fH_{m-2}$.
The subscript ``min'', indicating the minimality of
the space, is due to the following fact.
Because $\cH_{\mrm{A}}$ is the direct sum of
$\fH_m$ and an $m$-dimensional space $\fK_{\mrm{A}}$
spanned by the singular elements
$h_j\in\fH_{-m-2+2j}\setm\fH_{-m-1+2j}$, we have that
$\fK^{\min}_{\mrm{A}}\subseteq\fK\subseteq\fH_{-m}$,
where $\fK^{\min}_{\mrm{A}}:=\fK_{\mrm{A}}\mcap\fH_{m-2}$
is a minimal subset contained in $\fK_{\mrm{A}}$ in
the sense that $\fK_{\mrm{A}}\mcap\fK_{m-1}=\{0\}$.

Consider the domain restriction
$\Amax\vrt_{\cH^{\min}_{\mrm{A}}}$ to
$\cH^{\min}_{\mrm{A}}=\fH_m\dsum\fK^{\min}_{\mrm{A}}$
of $\Amax$. Let $\Bmax$ denote a linear relation in
$\cH_{\mrm{A}}$ defined by the componentwise sum of
(the graph of)
$\Amax\vrt_{\cH^{\min}_{\mrm{A}}}$ and
$\{0\}\times\cH^\bot_{\mrm{A}}$. Here $\cH^\bot_{\mrm{A}}$
denotes the orthogonal complement in $\cH_{\mrm{A}}$ of
$\cH^{\min}_{\mrm{A}}$, which is a subset of
$\fK_{\mrm{A}}$.
By the construction,
the adjoint $\Bmin:=\Bmax^*$ in $\cH_{\mrm{A}}$
is a linear relation given by the componentwise sum of
(the graph of)
$\Amin\vrt_{\cH^{\min}_{\mrm{A}}}$ and
$\{0\}\times\cH^\bot_{\mrm{A}}$.
Assuming only the invertibility and the Hermiticity of
$\cG_{\mrm{A}}$, the operator $\Amin$ differs from
$\Amin^\prime:=\Amax\vrt_{\ker\Gamma}$
(although $\dom\Amin=\dom\Amin^\prime$), \ie
$\Amin$ is not symmetric; the symmetry of
$\Amin=\Amin^\prime$ is ensured by \eqref{eq:1}.
Now the key point is that, without assumption
\eqref{eq:1}, but instead assuming
$[\cG_{\mrm{A}}]_{m-1,m}=[\cG_{\mrm{A}}]_{m,m-1}$
(the second condition in \eqref{eq:2}), it holds
\[
(\Amin-\Amin^\prime)(\dom\Amin\mcap\cH^{\min}_{\mrm{A}})
\subseteq\cH^\bot_{\mrm{A}}
\]
\ie $\Bmin$ is a symmetric linear relation in $\cH_{\mrm{A}}$.
By the same reasoning one shows that $\Bmin$ is also closed.
Sequentially,
one can apply the extension theory for $\Bmin$, as is
done for $\Amin$.

For $\cG_{\mrm{A}}$ as in \eqref{eq:1},
the Weyl function corresponding to a boundary triple
for $\Amax$ determined by $\Gamma$
is the sum of the Krein $Q$-function $q$ of
$\Lmin$ and a generalized Nevanlinna function $r$
(see \eg \cite[Section~4]{Behrndt11} for the terminology)
defined by
\[
r(z):=-\sum_{j=1}^m\frac{[\cG_{\mrm{A}}]_{mj}}{(z-z_1)^{m-j+1}}\,,
\quad z\in\bbC\setm\{z_1\}\,.
\]
Likewise,
for $\cG_{\mrm{A}}$ as in \eqref{eq:2},
the Weyl function corresponding to the boundary triple
for $\Bmax$, which is determined by restriction to $\dom\Bmax$
of $\Gamma$, is the sum of the same Krein $Q$-function $q$
and now a Nevanlinna function $\hat{r}$ defined by
\[
\hat{r}(z):=\frac{[\cG_{\mrm{A}}]_{mm}}{\hat{\Delta}-z}\,,
\quad z\in\bbC\setm\{\hat{\Delta}\}
\]
with some real number $\hat{\Delta}$.
The strict inequality $[\cG_{\mrm{A}}]_{mm}>0$
in \eqref{eq:2} is closely related to the fact that
the subspace
$\cH^{\min}_{\mrm{A}}=(\fH_m\dsum\fK^{\min}_{\mrm{A}},
[\cdot,\cdot]_{\mrm{A}})$ of $\cH_{\mrm{A}}$
is a Hilbert space iff $[\cG_{\mrm{A}}]_{mm}>0$.
Thus, for example, one may take $\cG_{\mrm{A}}$
as the Gram matrix of vectors $h_j$ generating $\fK_{\mrm{A}}$,
in which case $[\cG_{\mrm{A}}]_{jj^\prime}=
\braket{h_j,h_{j^\prime}}_{-m}$, and the conditions in
\eqref{eq:2} are all satisfied. In contrast, the so
defined $\cG_{\mrm{A}}$ does not satisfy \eqref{eq:1}.
We remark that, for $m=1$, we have $\hat{\Delta}=z_1$,
and hence $\hat{r}=r$, as it should follow from
$\cH^{\min}_{\mrm{A}}=\cH_{\mrm{A}}$.
We also remark that an analogous development of extension
theory for $\Bmin$
takes place in the peak model for singular perturbations,
\cf \cite{Jursenas18}.

Because the Weyl function $q+\hat{r}$ of $\Bmin$
is a (uniformly strict) Nevanlinna function, it follows from
\cite[Theorem~2.2]{Langer77}
that $q+\hat{r}$ is the Weyl function of some closed simple
symmetric operator, corresponding to a certain boundary triple.
Following the terminology in \cite{Dijksma18},
it is precisely in this sense what we mean by saying that
the A-model with mutually equal model parameters leads to
a Hilbert space model (of the function $q+\hat{r}$).
For example, a simple symmetric operator may be considered
as the operator of multiplication by an independent
variable in a reproducing kernel Hilbert space
induced by the Nevanlinna pair $(1,q+\hat{r})$;
see \eg
\cite[Theorem~6.1]{Behrndt09},
\cite[Theorem~4.10]{Behrndt11},
\cite[Remark~2.6]{Derkach09a}.

Having determined the extensions to $\cH_{\mrm{A}}$ of $\Lmin$
one then interprets singular perturbations of $L$
by means of the compressions to $\fH_m$
of their resolvents.
Thus, for $d=1$, $B_\Theta\in\mrm{Ext}(\Bmin)$,
$\Theta\in\bbC\mcup\{\infty\}$,
the compressed resolvent of $B_\Theta$ is represented
in the generalized sense according to
\[
P_{\fH_m}(B_\Theta-z)^{-1}\vrt_{\fH_m}=
(L-z)^{-1}+
\frac{\braket{g(\ol{z}),\cdot}
(L-z)^{-1}h_m}{\Theta-q(z)-\hat{r}(z)}
\]
for a suitable $z\in\res L$. Here
$P_{\fH_m}$ is a projection in $\cH_{\mrm{A}}$ onto $\fH_m$,
$g(\ol{z})\in\fH_{-m}\setm\fH_{-m+1}$ is the
eigenvector of $\Lmax$ corresponding to the eigenvalue $\ol{z}$
(in particular $h_1=g(z_1)$), and
$\braket{\cdot,\cdot}$ is the duality pairing between
$\fH_{-m}$ and $\fH_m$.
By the above resolvent formula one concludes
that the spectral properties of (super) singular perturbations
in the A-model with equal model parameters
can be described by Nevanlinna functions.

The reasoning behind the above mentioned
interpretation of singular perturbations is that
there exists a bijective correspondence between
Nevanlinna families and generalized resolvents
of $\Lmin$, and the correspondence is established
via a generalized Krein--Naimark resolvent formula.
Thus, to a rational Nevanlinna function $\hat{r}-\Theta$,
with a real $\Theta$, there corresponds a self-adjoint
extension $\wtB$ of $\Lmin$ in some larger Hilbert space
$\wtilde{\fH}\supseteq\fH_m$, and such that
$\wtB\mcap L=\Lmin$. For more details the
reader may refer to
\cite{Dijksma18,Derkach09,Derkach95,Derkach94,Derkach91}.
\section{A brief overview of the A-model with
equal model parameters}
Here we restate the main results from
\cite{Jursenas19,Dijksma05}. The main tools and terminology
used in the theory of boundary relations of symmetric operators
(or linear relations)
are as in
\cite{Derkach17,Derkach15,Hassi13,Derkach12,Behrndt11,Hassi09,Hassi07,Derkach95} and in references therein.

We consider a lower semibounded self-adjoint
operator $L$ in a Hilbert space $\fH_0$, and we let
$(\fH_n)_{n\in\bbZ}$ be the scale of Hilbert spaces associated
with $L$. The scalar product in $\fH_n$ is conjugate linear in
the first factor and is defined via the
scalar product $\braket{\cdot,\cdot}_0$ in $\fH_0$
according to
\[
\braket{\cdot,\cdot}_n:=
\braket{b_n(L)^{1/2}\cdot,b_n(L)^{1/2}\cdot}_0\,,\quad
b_n(L):=(L-z_1)^n
\]
for some fixed model parameter $z_1\in\res L\mcap\bbR$
($\res L$ denotes the resolvent set of $L$, and similarly
for other operators).
To $L=L_0$ one associates a self-adjoint operator
$L_n:=L\vrt_{\fH_{n+2}}$ in $\fH_n$, and satisfying
$L_{n+1}\subset L_n$ and $\res L_n=\res L$.
For the reasons just described
we sometimes omit the subscript $n$ in $L_n$.

Let us fix $m$, $d\in\bbN$.
Let $\{\vp_\sigma\in\fH_{-m-2}\setm\fH_{-m-1}\}$ be the family
of linearly independent functionals;
$\sigma$ ranges over an index set $\cS$ of cardinality $d$.
The symmetric restriction $L_{\min}$ of $L$ to
the domain of $f\in\fH_{m+2}$ such that
$\braket{\vp_\sigma,f}=0$, for all $\sigma$, is a
densely defined, closed, symmetric operator in $\fH_m$,
and has defect numbers $(d,d)$. It is also
essentially self-adjoint operator in $\fH_0$.
The duality pairing $\braket{\cdot,\cdot}$ is defined
via the $\fH_0$-scalar product in a usual way.
We also define a vector valued functional $\vp$
via
$\braket{\vp,\cdot}=(\braket{\vp_\sigma,\cdot})
\co\fH_{m+2}\lto\bbC^d$; hence
$\Lmin=L_m\vrt_{\{f\in\fH_{m+2}\vrt\braket{\vp,f}=0\}}$.

The triplet adjoint $\Lmax$ of $\Lmin$
corresponding to the Hilbert triple
$\fH_m\subset\fH_0\subset\fH_{-m}$ is the operator
extending $L_{-m+2}$ to the domain
$\fH_{-m+2}\dsum\fN_z(L_{\max})$ (direct sum)
for $z\in\res L$. The eigenspace
$\fN_z(\Lmax)$ ($:=\ker(\Lmax-z)$) is the
linear span of the elements $g_\sigma(z)$
defined in the generalized sense according to
\[
g_\sigma(z):=(L-z)^{-1}\vp_\sigma\in\fH_{-m}\setm\fH_{-m+1}\,.
\]

Define an $md$-dimensional linear space
\[
\fK_{\mrm{A}}:=\spn\{h_\alpha\vrt
\alpha=(\sigma,j)\in\cS\times J\}\,,
\quad
J:=\{1,2,\ldots,m\}
\]
spanned by the elements
\[
h_{\sigma j}:=b_j(L)^{-1}\vp_\sigma\in
\fH_{-m-2+2j}\setm\fH_{-m-1+2j}\,.
\]
From here it follows that
$\fK^{\min}_{\mrm{A}}\subseteq\fK_{\mrm{A}}\subseteq\fH_{-m}$
with
\[
\fK^{\min}_{\mrm{A}}:=\fK_{\mrm{A}}\mcap\fH_{m-2}
=h_m(\bbC^d)\,,
\quad
h_m(c):=\sum_{\sigma}c_\sigma h_{\sigma m}\,,
\quad
c=(c_\sigma)\in\bbC^d
\]
and that in particular
$\fK^{\min}_{\mrm{A}}=\fK_{\mrm{A}}$ for $m=1$.
Note that $\fK_{\mrm{A}}\mcap\fH_{m-1}=\{0\}$.

Because the system $\{h_\alpha\}$ is linearly independent,
the matrix
\[
\wtcG_{\mrm{A}}=([\wtcG_{\mrm{A}}]_{\alpha\alpha^\prime})
\in[\bbC^{md}]\,,\quad
[\wtcG_{\mrm{A}}]_{\alpha\alpha^\prime}:=
\braket{h_\alpha,h_{\alpha^\prime}}_{-m}
\]
is the Gram matrix of vectors generating $\fK_{\mrm{A}}$;
hence it is positive definite, Hermitian.
One establishes a bijective correspondence
\[
\fK_{\mrm{A}}\ni k\leftrightarrow d(k)=(d_\alpha(k))\in\bbC^{md}
\]
via
\[
k=\sum_\alpha d_\alpha(k)
h_\alpha\,,\quad
d(k)=\wtcG^{-1}_{\mrm{A}}\braket{h,k}_{-m}\,,
\quad
\braket{h,\cdot}_{-m}=(\braket{h_\alpha,\cdot}_{-m})\,.
\]
Here and in what follows $d(\cdot)$ is interpreted as
a (bounded) vector valued functional from $\fK_{\mrm{A}}$
to $\bbC^{md}$.

Let us define the matrix
\[
\wtcG^{\min}_{\mrm{A}}=
([\wtcG^{\min}_{\mrm{A}}]_{\sigma\sigma^\prime})
\in[\bbC^{d}]\,,\quad
[\wtcG^{\min}_{\mrm{A}}]_{\sigma\sigma^\prime}:=
\braket{h_{\sigma m},h_{\sigma^\prime m}}_{-m}
\]
which is the Gram matrix of vectors generating
$\fK^{\min}_{\mrm{A}}$. Thus $\wtcG^{\min}_{\mrm{A}}$
is also positive definite, Hermitian, and one therefore
establishes a bijective correspondence
\[
\fK^{\min}_{\mrm{A}}\ni h_m(c)\leftrightarrow
c\in\bbC^{d}
\]
via
\[
c=(\wtcG^{\min}_{\mrm{A}})^{-1}\braket{h_m,h_m(c)}_{-m}\,,
\quad
\braket{h_m,\cdot}_{-m}=(\braket{h_{\sigma m},\cdot}_{-m})\,.
\]
On the other hand,
because $\fK^{\min}_{\mrm{A}}\subseteq\fK$,
to each $k=h_m(c)\in\fK^{\min}_{\mrm{A}}$ there
corresponds $d(k)=\eta(c)\in\bbC^{md}$, where
\[
\eta(c):=(\delta_{jm}c_\sigma)\,.
\]

Consider an indefinite inner product space
\[
\cH_{\mrm{A}}:=(\fH_m\dsum\fK_{\mrm{A}},
[\cdot,\cdot]_{\mrm{A}})
\]
equipped with an indefinite metric
\[
[f+k,f^\prime+k^\prime]_{\mrm{A}}:=
\braket{f,f^\prime}_m+
\braket{d(k),\cG_{\mrm{A}}d(k^\prime)}_{\bbC^{md}}
\]
for $f$, $f^\prime\in\fH_m$ and $k$, $k^\prime\in\fK_{\mrm{A}}$.
The matrix $\cG_{\mrm{A}}=([\cG_{\mrm{A}}]_{\alpha\alpha^\prime})$
is called the Gram matrix of the A-model; it is
initially assumed to be invertible and Hermitian,
but otherwise arbitrary. Thus in particular
$\cG_{\mrm{A}}\neq0$.
Clearly if $\cG_{\mrm{A}}$ is
positive, then $\cH_{\mrm{A}}$ becomes a Hilbert
space. Otherwise $\cH_{\mrm{A}}$ is a Pontryagin space.

For an appropriate $\cG_{\mrm{A}}$,
the extensions to $\cH_{\mrm{A}}$ of $\Lmin$
are the restrictions to $\cH_{\mrm{A}}$ of the triplet
adjoint $\Lmax$.
Let
\[
\Amax:=\Lmax\mcap\cH^2_{\mrm{A}}\,.
\]
Here and in what follows operators are frequently
identified with their graphs.
The operator $\Amax$ admits the following representation:
\begin{align*}
\Amax=&
\{(f^\#+h_{m+1}(c)+k,L_mf^\#+z_1h_{m+1}(c)+\wtk)\vrt
f^\#\in\fH_{m+2}\,;
\\
&c\in\bbC^d\,;\,k,\wtk\in\fK_{\mrm{A}}\,;\,
d(\wtk)=\fM_dd(k)+\eta(c) \}\,.
\end{align*}
An element $h_{m+1}(c)\in\fH_m\setm\fH_{m+1}$ is
defined by
\[
h_{m+1}(c):=\sum_\sigma c_\sigma h_{\sigma,m+1}\,,
\quad
h_{\sigma,m+1}:=b_{m+1}(L)^{-1}\vp_\sigma\,.
\]
The matrix $\fM_d:=\fM\op\cdots\op \fM$ ($d$ times)
is the matrix direct sum of $d$ matrices
$\fM=(\fM_{jj^\prime})\in[\bbC^m]$ defined as follows:
For $m\geq2$
\[
\fM_{jj^\prime}:=
1_{J\setm\{m\}}(j)(\delta_{jj^\prime}z_1+
1_{J\setm\{1\}}(j^\prime)\delta_{j+1,j^\prime} )
+\delta_{jm}\delta_{j^\prime m}z_1
\]
for $j$, $j^\prime\in J$; here $1_X$ is the characteristic
function of a set $X$.
For $m=1$, $\fM:=z_1$.

By direct computation,
the boundary form of $A_{\max}$ is represented in the form
\begin{align*}
[f,\Amax g]_{\mrm{A}}-[\Amax f,g]_{\mrm{A}}=&
\braket{d(k),(\cG_\fM-\cG^*_\fM)d(k^\prime)}_{\bbC^{md}}
\\
&+\braket{\Gamma_0f,\Gamma_1g}_{\bbC^d}
-\braket{\Gamma_1f,\Gamma_0g}_{\bbC^d}\,,
\end{align*}
\[
\cG_\fM:=\cG_{\mrm{A}}\fM_d
\]
with
$f=f^\#+h_{m+1}(c)+k\in\dom A_{\max}$;
$g=g^\#+h_{m+1}(c^\prime)+k^\prime\in\dom A_{\max}$;
$f^\#,g^\#\in\fH_{m+2}$; $c,c^\prime\in\bbC^d$;
$k,k^\prime\in\fK_{\mrm{A}}$.
The operator
$\Gamma:=(\Gamma_0,\Gamma_1)$ from $\dom\Amax$ to
$\bbC^d\times\bbC^d$ is defined by
\begin{align*}
\Gamma_0(f^\#+h_{m+1}(c)+k):=&c\,,
\\
\Gamma_1(f^\#+h_{m+1}(c)+k):=&
\braket{\vp,f^\#}-[\cG_{\mrm{A}}d(k)]_m
\end{align*}
with
\[
[\cG_{\mrm{A}}d(k)]_m:=([\cG_{\mrm{A}}d(k)]_{\sigma m})
\in\bbC^d\,.
\]
In the next lemma we give a description of the adjoint
of $\Amax$ and, moreover, we show that $\Gamma$
is surjective.
By considering $\Gamma$ as a single-valued linear relation
from $\cH^2_{\mrm{A}}$ to $\bbC^{2d}$ with
$\dom\Gamma=\Amax$, \ie
\[
\Gamma=\{\bigl((f,\Amax f),(\Gamma_0f,\Gamma_1f)\bigr)\vrt
f\in\dom\Amax \}
\]
we recall that its Krein space adjoint
$\Gamma^{[*]}$ is a linear relation from $\bbC^{2d}$
to $\cH^2_{\mrm{A}}$, and it consists of
$\bigl((\chi,\chi^\prime),(g,g^\prime)\bigr)$
such that $(\forall f\in\dom\Amax)$
\begin{equation}
[f,g^\prime]_{\mrm{A}}-
[\Amax f,g]_{\mrm{A}}=
\braket{\Gamma_0f,\chi^\prime}_{\bbC^d}-
\braket{\Gamma_1f,\chi}_{\bbC^d}\,.
\label{eq:Gammaadj}
\end{equation}
\begin{lem}\label{lem:Amax}
Similar to $\Amax$, define the operator
$\Amax^\prime$ in $\cH_{\mrm{A}}$ by
\begin{align*}
\Amax^\prime:=&
\{(f^\#+h_{m+1}(c)+k,L_mf^\#+z_1h_{m+1}(c)+\wtk^\prime)\vrt
f^\#\in\fH_{m+2}\,;
\\
&c\in\bbC^d\,;\,k,\wtk^\prime\in\fK_{\mrm{A}}\,;\,
d(\wtk^\prime)=\cG^{-1}_{\mrm{A}}\cG^*_\fM d(k)+\eta(c) \}\,.
\end{align*}
The following statements hold:
\begin{itemize}
\item[$\mrm{(i)}$]
Consider $\Gamma=(\Gamma_0,\Gamma_1)$ as a single-valued
linear relation with $\dom\Gamma=\Amax$. Let
$\Gamma^{[*]}$ be its Krein space adjoint. Then
the inverse $(\Gamma^{[*]})^{-1}=(\Gamma_0,\Gamma_1)$
is a single-valued linear relation with
$\ran\Gamma^{[*]}=\Amax^\prime$.
Moreover, $\Gamma$ is closed and surjective.
\item[$\mrm{(ii)}$]
The adjoint in $\cH_{\mrm{A}}$ of a closed operator
$\Amax$ is the operator
\[
\Amin:=\Amax^*=\Amax^\prime\vrt_{\ker\Gamma}\,.
\]
\item[$\mrm{(iii)}$]
Define the operator
\[
\Amin^\prime:=\Amax\vrt_{\ker\Gamma}
\]
in $\cH_{\mrm{A}}$.
Then $\Amin^\prime$ is closed, and
its adjoint in $\cH_{\mrm{A}}$
is $\Amin^{\prime\,*}=\Amax^\prime$.
\end{itemize}
\end{lem}
\begin{proof}
First we remark that
\begin{equation}
\ran\cG^*_{\fM}\subseteq\ran\cG_{\mrm{A}}
\label{eq:A1}
\end{equation}
so that $\Amax^\prime$ is defined correctly.
The inclusion in \eqref{eq:A1} is equivalent to the statement
that
\begin{equation}
(\forall\xi\in\bbC^{md})\,
(\exists\xi^\prime\in\bbC^{md})\;
\cG^*_\fM\xi=\cG_{\mrm{A}}\xi^\prime\,.
\label{eq:A}
\end{equation}

For $m=1$, $\cG_\fM=z_1\cG_{\mrm{A}}$, so
$\xi^\prime=z_1\xi$ solves \eqref{eq:A}
for an arbitrary Hermitian $\cG_{\mrm{A}}$.

For $m\geq2$ we have
\begin{align*}
[\cG_\fM]_{\sigma j,\sigma^\prime j^\prime}=&
z_1
[\cG_{\mrm{A}}]_{\sigma j,\sigma^\prime j^\prime}+
1_{J\setm\{1\}}(j^\prime)
[\cG_{\mrm{A}}]_{\sigma j;\sigma^\prime, j^\prime-1}
\intertext{and hence}
[\cG^*_\fM]_{\sigma j,\sigma^\prime j^\prime}=&
z_1
[\cG_{\mrm{A}}]_{\sigma j,\sigma^\prime j^\prime}+
1_{J\setm\{1\}}(j)
[\cG_{\mrm{A}}]_{\sigma, j-1;\sigma^\prime j^\prime}\,.
\end{align*}
Then $\cG^*_\fM\xi=\cG_{\mrm{A}}\xi^\prime$ reads
\[
[\cG_{\mrm{A}}(\xi^\prime-z_1\xi)]_{\sigma j}=
1_{J\setm\{1\}}(j)[\cG_{\mrm{A}}\xi]_{\sigma,j-1}\,.
\]
Put
\[
\mathring{\cG}_{\mrm{A}}=
([\mathring{\cG}_{\mrm{A}}]_{\alpha\alpha^\prime})
\in[\bbC^{md}]\,,\quad
[\mathring{\cG}_{\mrm{A}}]_{\sigma j,\sigma^\prime j^\prime}
:=1_{J\setm\{1\}}(j)
[\cG_{\mrm{A}}]_{\sigma,j-1;\sigma^\prime j^\prime}\,.
\]
Then
\[
\ran\mathring{\cG}_{\mrm{A}}\subseteq
\ran\cG_{\mrm{A}}\quad\text{and}\quad
\cG_{\mrm{A}}(\xi^\prime-z_1\xi)=
\mathring{\cG}_{\mrm{A}}\xi
\]
and therefore
\[
\xi^\prime=(z_1+\cG^{-1}_{\mrm{A}}\mathring{\cG}_{\mrm{A}} )
\xi
\]
solves \eqref{eq:A} for an arbitrary invertible
Hermitian $\cG_{\mrm{A}}$.

(i)
Letting $g=g^\natural+k_g$ and
$g^\prime=g^{\prime\,\natural}+k_{g^\prime}$
in \eqref{eq:Gammaadj} for some
$g^\natural$, $g^{\prime\,\natural}\in\fH_m$
and $k_g$, $k_{g^\prime}\in\fK_{\mrm{A}}$,
and using that
\[
\braket{\braket{\vp,f^\#},\chi }_{\bbC^d}=
\braket{(L-z_1)f^\#,h_{m+1}(\chi)}_m
\]
and
\[
\braket{[\cG_{\mrm{A}}d(k)]_m,\chi}_{\bbC^d}=
\braket{d(k),\cX\chi}_{\bbC^{md}}
\]
with
\[
\cX=([\cX]_{\alpha\sigma})\in[\bbC^d,\bbC^{md}]\,,
\quad
[\cX]_{\alpha\sigma}:=[\cG_{\mrm{A}}]_{\alpha,\sigma^\prime m}
\]
we find that
$(\forall f^\#\in\fH_{m+2})$
$(\forall c\in\bbC^d)$
$(\forall k\in\fK_{\mrm{A}})$
\begin{align}
0=&\braket{f^\#,g^{\prime\,\natural}-z_1h_{m+1}(\chi)}_m
-\braket{Lf^\#,g^\natural-h_{m+1}(\chi)}_m
\nonumber \\
&+\braket{c,\braket{h_{m+1},g^{\prime\,\natural}-
z_1g^\natural }_m-[\cG_{\mrm{A}}d(k_g)]_m-\chi^\prime }_{\bbC^d}
\nonumber \\
&+\braket{d(k),\cG_{\mrm{A}}d(k_{g^\prime})-
\cG^*_\fM d(k_g)-\cX\chi }_{\bbC^{md}}
\label{eq:abc}
\end{align}
with
\[
\braket{h_{m+1},\cdot}_{m}=
(\braket{h_{\sigma, m+1},\cdot}_{m})\,.
\]
Thus it follows that
\[
g^\natural=g^\#+h_{m+1}(\chi)\,,\quad
g^\#\in\fH_{m+2}\,,\quad
g^{\prime\,\natural}=L_mg^\#+z_1h_{m+1}(\chi)
\]
and
\[
\chi^\prime=
\braket{h_{m+1},(L-z_1)g^\#}_m-[\cG_{\mrm{A}}d(k_g)]_m
=\braket{\vp,g^\#}-[\cG_{\mrm{A}}d(k_g)]_m
\]
and
\[
d(k_{g^\prime})=\cG^{-1}_{\mrm{A}}\cG^*_\fM d(k_g)+
\cG^{-1}_{\mrm{A}}\cX\chi\,,\quad
\cG^{-1}_{\mrm{A}}\cX\chi=\eta(\chi)\,.
\]
This shows that
\[
(\Gamma^{[*]})^{-1}=
\{\bigl((f,\Amax^\prime f),(\Gamma_0f,\Gamma_1f)\bigr)\vrt
f\in\dom\Amax \}\,.
\]
Because
$\ker\Gamma^{[*]}=\mul(\Gamma^{[*]})^{-1}=\{0\}$,
it follows that
$\ol{\ran}\Gamma=\ran\ol{\Gamma}=\bbC^{2d}$,
and it therefore remains to verify that $\Gamma$ is closed.

The closure $\ol{\Gamma}$ is the Krein space adjoint
of $\Gamma^{[*]}$. Thus it consists
$\bigl((g,g^\prime),(\chi,\chi^\prime)\bigr)
\in\cH^2_{\mrm{A}}\times\bbC^{2d}$
such that $(\forall f\in\dom\Amax)$
equation~\eqref{eq:abc} holds, but with
$\cG^*_\fM$ replaced by $\cG_\fM$. By repeating
the subsequent steps as above, one finds that
$\ol{\Gamma}=\Gamma$.

(ii)
The adjoint linear relation $\Amin$ consists of
$(g,g^\prime)\in\cH^2_{\mrm{A}}$ such that
\eqref{eq:abc} holds, but with $\chi=0=\chi^\prime$;
therefore it is the operator as stated in the lemma.

(iii)
By the arguments as in the proof of (i),
$\Amax^{\prime\,*}=\Amin^\prime$; thus $\Amin^\prime$
is a closed operator whose adjoint in $\cH_{\mrm{A}}$
is as stated in the lemma.
\end{proof}
For $m=1$,
the matrix $\cG_\fM=z_1\cG_{\mrm{A}}$ is
automatically Hermitian, while for $m\geq2$,
we have $\cG^*_\fM=\cG_\fM$ iff
\begin{equation}
\begin{split}
[\cG_{\mrm{A}}]_{\sigma j,\sigma^\prime j^\prime}=&
[\cG_{\mrm{A}}]_{\sigma j^\prime,\sigma^\prime j}\,,\quad
j,j^\prime\in J\,,
\\
[\cG_{\mrm{A}}]_{\sigma j,\sigma^\prime j^\prime}=&0\,,\quad
j\in J\setm\{m\}\,,\quad
j^\prime\in\{1,\ldots,m-j\}\,,
\\
[\cG_{\mrm{A}}]_{\sigma j,\sigma^\prime m}=&
[\cG_{\mrm{A}}]_{\sigma,j+1;\sigma^\prime,m-1}\,,\quad
j\in J\setm\{m\}\,.
\end{split}
\label{eq:GAcomm}
\end{equation}
Note that the entries of $\cG_{\mrm{A}}$ in \eqref{eq:GAcomm},
which are diagonal in $\sigma\in\cS$, are real numbers.
Note also that $\wtcG_{\mrm{A}}$ does
not satisfy \eqref{eq:GAcomm}, because
$[\wtcG_{\mrm{A}}]_{\sigma 1,\sigma 1}>0$.

For an Hermitian $\cG_\fM$ we have
$\Amax^\prime=\Amax$, $\Amin^\prime=\Amin$,
and $\Gamma$ is a unitary operator,
$\Gamma^{-1}=\Gamma^{[*]}$. Subsequently, the triple
$(\bbC^d,\Gamma_0,\Gamma_1)$ is a boundary triple
for the adjoint $\Amax=\Amin^*$ of a densely defined, closed,
symmetric operator $\Amin$. An extension
$A_\Theta\in\mrm{Ext}(\Amin)$ of $\Amin$, \ie
an operator satisfying
$\Amin\subseteq A_\Theta\subseteq\Amax$, is parametrized
by a linear relation $\Theta$ in $\bbC^d$ according to
\[
\dom A_\Theta=\{f\in\dom\Amax\vrt\Gamma f\in\Theta\}\,.
\]
In particular, $A_\Theta$ is self-adjoint in $\cH_{\mrm{A}}$
iff $\Theta$ is self-adjoint in $\bbC^d$, because
the adjoint $A^*_\Theta$ in $\cH_{\mrm{A}}$ of $A_\Theta$
is given by $A_{\Theta^*}$, where $\Theta^*$ is the adjoint
in $\bbC^d$ of $\Theta$. The Krein--Naimark resolvent
formula for $A_\Theta$ reads
\[
(A_\Theta-z)^{-1}=
(A_0-z)^{-1}+\gamma_{\Gamma}(z)
(\Theta-M_{\Gamma}(z))^{-1}\gamma_{\Gamma}(\ol{z})^*
\]
for $z\in\res A_0\mcap \res A_\Theta$.
The self-adjoint operator $A_0$ corresponds to
the self-adjoint linear relation $\{0\}\times\bbC^d$ in $\bbC^d$,
and its resolvent is given by
\[
(A_0-z)^{-1}(f+k)=(L_m-z)^{-1}f+
\sum_\alpha[(\fM_d-z)^{-1}d(k)]_\alpha h_\alpha
\]
for $f\in\fH_m$, $k\in\fK_{\mrm{A}}$,
and $z\in\res A_0=\res L\setm\{z_1\}$.
The $\gamma$-field $\gamma_{\Gamma}$ and the Weyl
function $M_{\Gamma}$ corresponding to
$(\bbC^d,\Gamma_0,\Gamma_1)$ are described by
\[
\gamma_{\Gamma}(z)\bbC^d=\fN_z(A_{\max})
=\{\sum_\sigma c_\sigma F_\sigma(z)\vrt
c_\sigma\in\bbC\}\,,\quad
F_\sigma(z):=\frac{g_\sigma(z)}{(z-z_1)^m}
\]
and
\[
M_{\Gamma}(z)=q(z)+r(z)\quad\text{on}\quad\bbC^d
\]
for $z\in\res A_0$. The Krein $Q$-function $q$
of $L_{\min}$ is defined by
\[
q(z)=([q(z)]_{\sigma\sigma^\prime})\in[\bbC^d]\,,\quad
[q(z)]_{\sigma\sigma^\prime}:=
(z-z_1)\braket{\vp_\sigma,(L-z)^{-1}
h_{\sigma^\prime,m+1}}
\]
for $z\in\res L$,
and the generalized Nevanlinna function $r$ is defined by
\[
r(z)=([r(z)]_{\sigma\sigma^\prime})\in[\bbC^d]\,,\quad
[r(z)]_{\sigma\sigma^\prime}:=-\sum_j
\frac{[\cG_{\mrm{A}}]_{\sigma m,\sigma^\prime j}}{
(z-z_1)^{m-j+1}}
\]
for $z\in\bbC\setm\{z_1\}$.

The compressed resolvent of $A_\Theta$ is represented
in the generalized sense according to
\begin{align}
P_{\fH_m}(A_\Theta-z)^{-1}\vrt_{\fH_m}=&
(L-z)^{-1}
\nonumber \\
&+
\sum_\sigma
[(\Theta-M_{\Gamma}(z))^{-1}
\braket{\vp,(L-z)^{-1}\cdot}]_\sigma
(L-z)^{-1}h_{\sigma m}
\label{eq:comp}
\end{align}
for $z\in\res A_0\mcap\res A_\Theta$.
As expected, in the A-model with equal model parameters
the spectral properties of singular rank-$d$ perturbations
of class $\fH_{-4}$ or higher
are described by a generalized Nevanlinna function $M_{\Gamma}$.
\section{Extensions which are linear relations}
Let $j_\star\in J$; then
\[
\fH_m\mcap\fH_{-m-2+2j_\star}=\fH_m
\]
while
\[
\fK_{\mrm{A}}\mcap\fH_{-m-2+2j_\star}=
\spn\{h_{\sigma j}\vrt(\sigma,j)\in\cS\times
\{j_\star,\ldots,m\} \}
\]
is a $d(m-j_\star+1)$-dimensional linear space.
Choosing $j_\star=m$ we therefore construct a
$d$-dimensional subspace $\fK^{\min}_{\mrm{A}}$
of $\fK_{\mrm{A}}$,
which is minimal in the sense that
$\fK_{\mrm{A}}\mcap\fH_{m-1}=\{0\}$.
Let
\[
\cH^{\min}_{\mrm{A}}:=(\fH_m\dsum\fK^{\min}_{\mrm{A}},
[\cdot,\cdot]_{\mrm{A}})\,.
\]
That is, $\cH^{\min}_{\mrm{A}}$ is a subspace of
$\cH_{\mrm{A}}$ equipped with an indefinite metric
\[
[f+h_m(c),f^\prime+h_m(c^\prime)]_{\mrm{A}}=
\braket{f,f^\prime}_m+
\braket{c,\cG^{\min}_{\mrm{A}}c^\prime}_{\bbC^d}
\]
for $f$, $f^\prime\in\fH_m$ and
$c$, $c^\prime\in\bbC^d$. The matrix
\[
\cG^{\min}_{\mrm{A}}=
([\cG^{\min}_{\mrm{A}}]_{\sigma\sigma^\prime})
\in[\bbC^{d}]\,,\quad
[\cG^{\min}_{\mrm{A}}]_{\sigma\sigma^\prime}:=
[\cG_{\mrm{A}}]_{\sigma m,\sigma^\prime m}
\]
where, as previously, $\cG_{\mrm{A}}$ is the Gram
matrix of the A-model; \ie it is invertible and Hermitian.
The matrix $\cG^{\min}_{\mrm{A}}$ is Hermitian, and
the space $\cH^{\min}_{\mrm{A}}$ is a Hilbert space iff
an Hermitian $\cG^{\min}_{\mrm{A}}$ is positive definite.
In this case $\cH^{\min}_{\mrm{A}}$ becomes a subspace
of the positive subspace of the Pontryagin space
$\cH_{\mrm{A}}$.
\begin{lem}\label{lem:cHbot}
Let $\cH^\bot_{\mrm{A}}$ denote the orthogonal complement in
$\cH_{\mrm{A}}$ of $\cH^{\min}_{\mrm{A}}$. Then:
\begin{itemize}
\item[$\mrm{(i)}$]
$\cH^\bot_{\mrm{A}}$ is a subset of $\fK_{\mrm{A}}$
given by
\[
\cH^\bot_{\mrm{A}}=
\{k\in\fK_{\mrm{A}}\vrt
[\cG_{\mrm{A}}d(k)]_m=0 \}\,.
\]
\item[$\mrm{(ii)}$]
Assume that
\begin{equation}
[\cG_{\mrm{A}}]_{\sigma, m-1;\sigma^\prime m}=
[\cG_{\mrm{A}}]_{\sigma m;\sigma^\prime, m-1}\,,\quad
\sigma,\sigma^\prime\in\cS
\label{eq:A2}
\end{equation}
if $m\geq2$.
Then
\[
(\Amax^\prime-\Amax)\fK^{\min}_{\mrm{A}}\subseteq
\cH^\bot_{\mrm{A}}\,.
\]
\end{itemize}
\end{lem}
Recall that $\Amax^\prime=\Amax$ if $m=1$.
\begin{proof}
(i)
$\cH^\bot_{\mrm{A}}$ is the set of
$g+k\in\fH_m\dsum\fK_{\mrm{A}}$ such that
$(\forall f\in\fH_m)$ $(\forall c\in\bbC^d)$
\[
0=\braket{f,g}_m+\braket{\eta(c),\cG_{\mrm{A}}d(k)}_{\bbC^d}
=\braket{f,g}_m+\braket{c,[\cG_{\mrm{A}}d(k)]_m}_{\bbC^d}\,;
\]
hence such that $g=0$ and $[\cG_{\mrm{A}}d(k)]_m=0$.

(ii)
We have $(\forall c\in\bbC^d)$
\[
(\Amax^\prime-\Amax)h_m(c)=\wtk^{\prime\prime}\in
\fK_{\mrm{A}}\,,\quad
d(\wtk^{\prime\prime})=
(\cG^{-1}_{\mrm{A}}\cG^*_\fM-\fM_d)\eta(c)\,.
\]
Then $(\forall\sigma\in\cS)$
\begin{align*}
[\cG_{\mrm{A}}d(\wtk^{\prime\prime})]_{\sigma m}=&
[(\cG^*_\fM-\cG_\fM)\eta(c)]_{\sigma m}
\\
=&\sum_{\sigma^\prime}
([\cG_{\mrm{A}}]_{\sigma,m-1;\sigma^\prime m}-
[\cG_{\mrm{A}}]_{\sigma m;\sigma^\prime,m-1} )c_{\sigma^\prime}\,.
\end{align*}
By hypothesis one therefore sees that
$\wtk^{\prime\prime}\in\cH^\bot_{\mrm{A}}$.
\end{proof}
Define a linear relation $\Bmax$ in $\cH_{\mrm{A}}$
by
\[
\Bmax:=\Amax\vrt_{\dom\Amax\mcap\cH^{\min}_{\mrm{A}}}\hsum
(\{0\}\times\cH^\bot_{\mrm{A}})
\]
(the componentwise sum),
where
\[
\Amax\vrt_{\dom\Amax\mcap\cH^{\min}_{\mrm{A}}}=
\Amax\mcap(\cH^{\min}_{\mrm{A}}\times\cH_{\mrm{A}})
\]
is the domain restriction to
$\dom\Amax\mcap\cH^{\min}_{\mrm{A}}$
of $\Amax$. Let also
\[
\Bmin:=\Bmax^*
\]
be the adjoint in $\cH_{\mrm{A}}$ of $\Bmax$.

For $m=1$ we have $\cH^{\min}_{\mrm{A}}=\cH_{\mrm{A}}$
and $\cH^\bot_{\mrm{A}}=\{0\}$, corresponding to
$\fK^{\min}_{\mrm{A}}=\fK_{\mrm{A}}$.
In this case $\Bmax=\Amax$ and $\Bmin=\Amin$ are operators.
But for $m\geq2$, $\Bmax$ has a nontrivial multivalued
part $\mul\Bmax=\cH^\bot_{\mrm{A}}$. The multivalued part of
$\Bmin$ is also $\cH^\bot_{\mrm{A}}$, which is seen
from $\mul\Bmin=(\dom\Bmax)^{\bot}$ and using that
$\fH_{m+2}$ is dense in $\fH_m$.
We have, moreover, the next lemma.
\begin{lem}\label{lem:Bmax}
Assume \eqref{eq:A2} if $m\geq2$. Then
\[
\Bmax=\Amax^\prime\vrt_{\dom\Amax\mcap\cH^{\min}_{\mrm{A}}}\hsum
(\{0\}\times\cH^\bot_{\mrm{A}})
\]
and
\begin{align*}
\Bmin=&\Amin\vrt_{\dom\Amin\mcap\cH^{\min}_{\mrm{A}}}\hsum
(\{0\}\times\cH^\bot_{\mrm{A}})
\\
=&\Amin^\prime\vrt_{\dom\Amin\mcap\cH^{\min}_{\mrm{A}}}\hsum
(\{0\}\times\cH^\bot_{\mrm{A}})\,.
\end{align*}
Moreover, $\Bmin$ is a closed symmetric linear relation in
$\cH_{\mrm{A}}$, whose adjoint in $\cH_{\mrm{A}}$
is the linear relation $\Bmin^*=\Bmax$.
\end{lem}
\begin{proof}
For $m=1$ the statements of the lemma follow
from Lemma~\ref{lem:Amax}, so in what follows
we let $m\geq2$.

The representation of $\Bmax$, as stated, is due to
Lemma~\ref{lem:cHbot}.
The adjoint of $\Bmax$ is given by
(recall \eg \cite[Lemma~2.6]{Hassi09})
\[
\Bmin=
(\Amax\vrt_{\dom\Amax\mcap\cH^{\min}_{\mrm{A}}})^*\mcap
(\{0\}\times\cH^\bot_{\mrm{A}})^*
\]
with
\[
(\{0\}\times\cH^\bot_{\mrm{A}})^*=
\cH^{\min}_{\mrm{A}}\times\cH_{\mrm{A}}\,.
\]
Because $\Amax$ and $\cH^{\min}_{\mrm{A}}$ are closed,
by the same argument we also get that
\begin{align*}
(\Amax\vrt_{\dom\Amax\mcap\cH^{\min}_{\mrm{A}}})^*
=&
[\Amax\mcap(\cH^{\min}_{\mrm{A}}\times\cH_{\mrm{A}})]^*
\\
=&\ol{\Amin\hsum (\{0\}\times\cH^\bot_{\mrm{A}})}\,.
\end{align*}
Because
\[
\Amin^*\hsum (\{0\}\times\cH^\bot_{\mrm{A}})^*=
\Amax\hsum(\cH^{\min}_{\mrm{A}}\times\cH_{\mrm{A}})=
\cH^2_{\mrm{A}}
\]
is a closed linear relation, we have by
\cite[Lemma~2.10]{Hassi09} that
\[
\ol{\Amin\hsum (\{0\}\times\cH^\bot_{\mrm{A}})}=
\Amin\hsum (\{0\}\times\cH^\bot_{\mrm{A}})
\]
is also closed. Combining all together we deduce
the first representation of $\Bmin$ as stated in the lemma.
By using this representation and noting that
$\Amin\subseteq\Amax^\prime$ and
$\Amin^\prime\subseteq\Amax$ (Lemma~\ref{lem:Amax}),
we deduce also the second formula for $\Bmin$
by applying Lemma~\ref{lem:cHbot}.
The computation of the adjoint $\Bmin^*$ uses
the same arguments as that of $\Bmax^*$,
and one concludes that $\Bmin$ is a closed
symmetric linear relation.
\end{proof}
The boundary value space of $\Bmin$ is
characterized by the next theorem.
\begin{thm}\label{thm:BVSBmin}
Assume \eqref{eq:A2} if $m\geq2$, and let
$\cG^{\min}_{\mrm{A}}$ be positive definite.
Define the operator
$\Gamma^\prime:=
(\Gamma^\prime_0,\Gamma^\prime_1)\co\Bmax\lto\bbC^{2d}$
by
\[
\Gamma^\prime_0\whf:=c\,,\quad
\Gamma^\prime_1\whf:=\braket{\vp,f^\#}
-\cG^{\min}_{\mrm{A}}\chi
\]
for $\whf=(f,f^\prime)\in\Bmax$; that is
\begin{align*}
f=&f^\#+h_{m+1}(c)+h_m(\chi)\,,\quad
f^\#\in\fH_{m+2}\,,\quad
c,\chi\in\bbC^d\,,
\\
f^\prime=&L_mf^\#+z_1h_{m+1}(c)+\wtk+k_\bot\,,\quad
\wtk\in\fK_{\mrm{A}}\,,\quad
k_\bot\in\cH^\bot_{\mrm{A}}\,,
\end{align*}
\[
d(\wtk)=\fM_d \eta(\chi)+\eta(c)\,.
\]
Then $(\bbC^d,\Gamma^\prime_0,\Gamma^\prime_1)$
is a boundary triple for $\Bmax$. The corresponding
$\gamma$-field $\gamma_{\Gamma^\prime}$
and the Weyl function $M_{\Gamma^\prime}$ are
bounded analytic operator functions given by
\[
\gamma_{\Gamma^\prime}(z)\bbC^d=
\fN_z(\Bmax)=
\{(L-z)^{-1}h_m(c)+h_m(\chi)\vrt \chi=(z-\hat{\Delta})^{-1}c\,;\,
c\in\bbC^d \}\,,
\]
\[
\hat{\Delta}:=(\cG^{\min}_{\mrm{A}})^{-1}\Delta
\in[\bbC^d]\,,\quad
\Delta=(\Delta_{\sigma\sigma^\prime})=
\Delta^*\in[\bbC^d]\,,
\]
\[
\Delta_{\sigma\sigma^\prime}:=
[\cG_\fM]_{\sigma m,\sigma^\prime m}
=z_1[\cG^{\min}_{\mrm{A}}]_{\sigma\sigma^\prime}+
1_{\bbN_{\geq2}}(m)
[\cG_{\mrm{A}}]_{\sigma,m-1;\sigma^\prime m}
\]
and
\[
M_{\Gamma^\prime}(z)=q(z)+\hat{r}(z)\,,\quad
\hat{r}(z):=\cG^{\min}_{\mrm{A}}(\hat{\Delta}-z)^{-1}
\]
for $z\in\res L\mcap\res\hat{\Delta}$.
Moreover, $M_{\Gamma^\prime}$ is a uniformly strict
Nevanlinna function.
\end{thm}
\begin{proof}
Step 1.
In this step we argue as in the proof of Lemma~\ref{lem:Bmax}.
Consider $\Gamma^\prime$ as a single-valued
linear relation with $\dom\Gamma^\prime=\Bmax$:
\[
\Gamma^\prime=\{(\whf,(\Gamma^\prime_0\whf,
\Gamma^\prime_1\whf))\vrt\whf\in\Bmax\}\,.
\]
Likewise,
consider $\Gamma$ as a single-valued
linear relation with $\dom\Gamma=\Amax$:
\[
\Gamma=\{(\whf,(\Gamma_0\whf,
\Gamma_1\whf))\vrt\whf\in\Amax\}\,.
\]
Then by definition
\[
\Gamma^\prime=(\Gamma\mcap\fM)\hsum\fN\,,\quad
\fM:=(\cH^{\min}_{\mrm{A}}\times\cH_{\mrm{A}})
\times\bbC^{2d}\,,\quad
\fN:=(\{0\}\times\cH^\bot_{\mrm{A}})\times\{0\}\,.
\]
Then the Krein space adjoint of $\Gamma^\prime$ is given by
\[
(\Gamma^\prime)^{[*]}=(\Gamma\mcap\fM)^{[*]}\mcap
\fN^{[*]}\,,\quad
\fN^{[*]}=\bbC^{2d}\times
(\cH^{\min}_{\mrm{A}}\times\cH_{\mrm{A}})=\fM^{-1}\,.
\]
Because $\fM$ is a closed linear relation,
and so is $\Gamma$ by Lemma~\ref{lem:Amax}(i),
the Krein space adjoint of $\Gamma\mcap\fM$
is given by
\[
(\Gamma\mcap\fM)^{[*]}=\ol{\Gamma^{[*]}\hsum\fM^{[*]}}\,,
\quad
\fM^{[*]}=\{0\}\times(\{0\}\times\cH^\bot_{\mrm{A}})
=\fN^{-1}\,.
\]
Because $\Gamma\hsum\fM=\cH^2_{\mrm{A}}\times\bbC^{2d}$,
it follows that
\[
(\Gamma\mcap\fM)^{[*]}=\Gamma^{[*]}\hsum
\fN^{-1}=[(\Gamma^{[*]})^{-1}\hsum\fN]^{-1}
\]
and therefore
\[
(\Gamma^\prime)^{[*]}=
[(\Gamma^{[*]})^{-1}\hsum\fN]^{-1}\mcap\fM^{-1}=
\{[(\Gamma^{[*]})^{-1}\mcap\fM]\hsum\fN\}^{-1}\,.
\]
By applying Lemma~\ref{lem:Amax}(i) and
Lemma~\ref{lem:cHbot}(ii), this leads to
$(\Gamma^\prime)^{[*]}=(\Gamma^\prime)^{-1}$.

Since $\Gamma^\prime$ is single-valued, unitary, and
with closed domain, we conclude that
$\Gamma^\prime$ is surjective, and then the triple
$(\bbC^d,\Gamma^\prime_0,\Gamma^\prime_1)$
is a boundary triple for $\Bmax$.

Step 2.
We compute the eigenspace of $\Bmax$.
For $f\in\fN_z(\Bmax)$, $z\in\bbC$, we have
\[
0=(L-z)f^\#+(z_1-z)h_{m+1}(c)\,,\quad
0=\wtk+k_\bot-zh_m(\chi)\,.
\]
Then, for $z\in\res L$, the first equation leads to
\[
f^\#=(z-z_1)(L-z)^{-1}h_{m+1}(c)=
-h_{m+1}(c)+(L-z)^{-1}h_m(c)
\]
The second equation implies that
\[
0=d(\wtk)+d(k_\bot)-z\eta(\chi)\quad\text{or else}\quad
d(k_\bot)=(z-\fM_d)\eta(\chi)-\eta(c)\,.
\]
Because $k_\bot\in\cH^\bot_{\mrm{A}}$, we have that
$[\cG_{\mrm{A}}d(k_\bot)]_m=0$; hence
\[
0=\cG^{\min}_{\mrm{A}}(z\chi-c)-
[\cG_{\fM}\eta(\chi)]_m\,,\quad
[\cG_{\fM}\eta(\chi)]_m=\Delta\chi\,.
\]
Because by hypothesis an Hermitian $\cG^{\min}_{\mrm{A}}$ is
positive definite, the latter shows that
\[
0=(z-\hat{\Delta})\chi-c\quad\Rightarrow\quad
\chi=(z-\hat{\Delta})^{-1}c\,,\quad z\in\res\hat{\Delta}\,.
\]

Step 3.
By definition $\gamma_{\Gamma^\prime}(z)c=f\in\fN_z(\Bmax)$;
thus by Step 2, we get $\gamma_{\Gamma^\prime}(z)$
as claimed. Again by definition
$M_{\Gamma^\prime}(z)c=\Gamma^\prime_1(f,zf)$,
$f\in\fN_z(\Bmax)$;
thus by Step 2, we get $M_{\Gamma^\prime}(z)$
as stated in the lemma.

Step 4.
Because $q$ is the Weyl function corresponding
to the boundary triple
$(\bbC^d,\mathring{\Gamma}_0,\mathring{\Gamma}_1)$
for the adjoint in $\fH_m$ of $\Lmin$, where
(\cite[Corollary~7.4]{Jursenas19})
\[
\mathring{\Gamma}_{0}(f^\#+h_{m+1}(c)):=c\,,\quad
\mathring{\Gamma}_{1}(f^\#+h_{m+1}(c)):=
\braket{\vp,f^\#}\,,
\]
we have by \eg \cite[Theorem~1.4]{Derkach17}
that $q$ is a uniformly strict Nevanlinna function.

By hypothesis imposed on $\cG_{\mrm{A}}$, the matrix
$\Delta$ is Hermitian, so the matrix function $\hat{r}$
is symmetric with respect to the real axis,
$\hat{r}(z)^*=\hat{r}(\ol{z})$, $z\in\res\hat{\Delta}$.
We prove that $\res\hat{\Delta}\supseteq\bbC\setm\bbR$.
Because $\hat{r}$ is analytic on $\res\hat{\Delta}$,
and moreover the matrix
\[
\frac{\Im\hat{r}(z)}{\Im z}=AB(z)\,,
\quad \Im z\neq0\,,
\]
\[
A:=(\cG^{\min}_{\mrm{A}})^{-2}>0\,,\quad
B(z):=\hat{r}(z)^*(\cG^{\min}_{\mrm{A}})^{-1}
\hat{r}(z)>0
\]
is similar to the positive definite matrix
$B(z)^{1/2}AB(z)^{1/2}$,
this would imply that $\hat{r}$ is a uniformly
strict Nevanlinna function.

The spectrum of $\hat{\Delta}$ consists of
$z\in\bbC$ such that the determinant
$\det(\hat{\Delta}-z)=0$. Because $\hat{\Delta}$
is the product of two Hermitian matrices, using
their spectral decompositions we get that $z$
solves $\det(Y-z)=0$, where the matrix
$Y:=\Lambda^{-1}X$, $\Lambda$ is the positive
definite diagonal matrix with the eigenvalues of
$\cG^{\min}_{\mrm{A}}$ on its diagonal,
and $X$ is an Hermitian matrix. Because
$Y=\Lambda^{-1/2}Y^\prime\Lambda^{1/2}$ is similar
to an Hermitian matrix
$Y^\prime:=\Lambda^{-1/2}X\Lambda^{-1/2}$, we get that
$z$ is an eigenvalue of $Y^\prime$, and hence belongs to $\bbR$.
Consequently, $\res\hat{\Delta}\supseteq\bbC\setm\bbR$
as claimed.

The sum $M_{\Gamma^\prime}$ of two uniformly strict
Nevanlinna functions $q$ and $\hat{r}$ is itself of
the same class, as can be deduced from
\cite[Lemma~2.6]{Behrndt13}
\cite[Proposition~3.2]{Daho85},
and this accomplishes the proof of the theorem.
\end{proof}
Under assumptions of Theorem~\ref{thm:BVSBmin},
consider $\Gamma^\prime$ as a (unitary) single-valued
linear relation with $\dom\Gamma^\prime=\Bmax$.
According to \cite[Theorem~4.8]{Behrndt11}, if $\Gamma^\prime$
is minimal, \ie if the closed linear span
\[
\cH_s:=\ol{\spn}\{\fN_z(\Bmax)\vrt z\in\reg\Bmin\}
\]
($\reg\Bmin$ is the regularity domain of $\Bmin$;
see \eg \cite[Eq.~(6.14)]{Azizov89}) coincides with
$\cH_{\mrm{A}}$, then $M_{\Gamma^\prime}$ must be
a generalized Nevanlinna function with a generally
nontrivial number $\kappa$ of negative squares
(where $\kappa$ is equal to the rank
of indefiniteness of the Pontryagin space $\cH_{\mrm{A}}$).
Recall that $\cH_s=\cH_{\mrm{A}}$ means also that
a closed symmetric linear relation $\Bmin$ is simple.
If, however, $\Gamma^\prime$ is not minimal,
then $M_{\Gamma^\prime}$ is a generalized Nevanlinna
function with $\kappa^\prime\leq\kappa$ negative squares.
By Theorem~\ref{thm:BVSBmin} we have $\kappa^\prime=0$,
and by the next proposition this corresponds to the fact that
$\Gamma^\prime$ is not a minimal boundary relation for $\Bmax$
for at least $m\geq2$, unless $\cH^\bot_{\mrm{A}}=\{0\}$;
if the latter holds then by our hypothesis on $\cG_{\mrm{A}}$
the space $\cH_{\mrm{A}}=\cH^{\min}_{\mrm{A}}$ is a Hilbert
space (for all $m\geq1$), and hence $\kappa=0$.
\begin{thm}
Under assumptions of Theorem~\ref{thm:BVSBmin},
$\emptyset\neq\cH_s\subseteq\cH^{\min}_{\mrm{A}}$.
Moreover, if the only solutions $f\in\fH_m$ and $\chi\in\bbC^d$
to
\begin{equation}
(\forall z\in\bbC\setm\bbR)\;
\braket{\vp,(L-z)^{-1}f}=
\hat{r}(z)\chi
\label{eq:A3}
\end{equation}
are $f=0$ and $\chi=0$,
then $\cH_s=\cH^{\min}_{\mrm{A}}$.
\end{thm}
\begin{proof}
First we prove the next lemma.
\begin{lem}\label{lem:BVSBmin-1}
$(\forall k\in\fK_{\mrm{A}})$
$(\exists\chi\in\bbC^d)$ $(\exists k_\bot\in\cH^\bot_{\mrm{A}})$
$d(k)=\eta(\chi)+d(k_\bot)$.
\end{lem}
\begin{proof}
Because every $f\in\cH_{\mrm{A}}$ is of the form
$f=f^\prime+k_\bot$, for some $f^\prime\in\cH^{\min}_{\mrm{A}}$
and $k_\bot\in\cH^\bot_{\mrm{A}}$,
we have that $f^\prime=f^{\prime\prime}+h_m(\chi)$, for
some $f^{\prime\prime}\in\fH_m$ and $\chi\in\bbC^d$.
Choosing $f^{\prime\prime}=0$ the claim follows.
\end{proof}
That $\cH_s$ is nonempty follows from the
following lemma (recall that
$\res L\mcap\res\hat{\Delta}\supseteq\bbC\setm\bbR$).
\begin{lem}\label{lem:regBmin}
$\reg\Bmin\supseteq\res L\mcap\res\hat{\Delta}$.
\end{lem}
\begin{proof}
We show that, for $z\in\res L\mcap\res\hat{\Delta}$,
the eigenspace
$\fN_z(\Bmin)=\{0\}$ and the range
$\ran(\Bmin-z)$ is closed, from which the statement of
the lemma follows.

The linear relation $\Bmin=\ker\Gamma^\prime$
explicitly reads
\begin{align*}
\Bmin=&\{(f^\#+h_m(\chi),L_mf^\#+\wtk+k_\bot)\vrt
f^\#\in\fH_{m+2}\,;\chi\in\bbC^d\,;
\\
&k_\bot\in\cH^\bot_{\mrm{A}}\,;\,\wtk\in\fK_{\mrm{A}}\,;\,
d(\wtk)=\fM_d\eta(\chi)\,;\,
\braket{\vp,f^\#}=\cG^{\min}_{\mrm{A}}\chi\}\,.
\end{align*}
Therefore $f\in\fN_z(\Bmin)$ solves
\[
0=(L_m-z)f^\#\,,\quad
0=(\hat{\Delta}-z)\chi\,,\quad
\braket{\vp,f^\#}=\cG^{\min}_{\mrm{A}}\chi\,.
\]
Since $z\in\res L_m=\res L$, this leads to $f=0$.

By applying Lemma~\ref{lem:BVSBmin-1}
\[
\wtk=h_m(\hat{\Delta}\chi)+k^\prime_\bot\,,\quad
k^\prime_\bot\in\cH^\bot_{\mrm{A}}\,.
\]
Therefore the range
\begin{align*}
\ran(\Bmin-z)=&
\{(L_m-z)f^\#+h_m((\hat{\Delta}-z)\chi)
+k_\bot\vrt f^\#\in\fH_{m+2}\,;\,
\chi\in\bbC^d\,;
\\
&k_\bot\in\cH^\bot_{\mrm{A}}\,;\,
\braket{\vp,f^\#}=\cG^{\min}_{\mrm{A}}\chi \}
\\
&(z\in\bbC)
\\
=&
\{(L_m-z)f^\#+h_m(\chi)+k_\bot\vrt f^\#\in\fH_{m+2}\,;\,
\chi\in\bbC^d\,;
\\
&k_\bot\in\cH^\bot_{\mrm{A}}\,;\,
\braket{\vp,f^\#}=\hat{r}(z)\chi \}
\\
&(z\in\res\hat{\Delta})\,.
\end{align*}

On the other hand, the closure
$\ol{\ran}(\Bmin-z)$, $\ol{z}\in\res L\mcap\res\hat{\Delta}$,
is the orthogonal complement in $\cH_{\mrm{A}}$
of $\fN_{\ol{z}}(\Bmax)$; hence
\begin{align*}
\ol{\ran}(\Bmin-z)=&
\{f+k\in\fH_m\dsum\fK_{\mrm{A}}\vrt
(\forall c\in\bbC^d)\,
\\
&0=\braket{f,(L-\ol{z})^{-1}h_m(c)}_m+
\braket{d(k),\cG_{\mrm{A}}\eta(\chi)}_{\bbC^{md}}\,;
\\
&\chi=(\ol{z}-\hat{\Delta})^{-1}c \}\,.
\end{align*}
Note that $\res L\mcap\res\hat{\Delta}\supseteq\bbC\setm\bbR$
implies that also $z\in\res L\mcap\res\hat{\Delta}$.

We have
\[
\braket{f,(L-\ol{z})^{-1}h_m(c)}_m=
\braket{\braket{\vp,(L-z)^{-1}f},c }_{\bbC^d}\,,
\]
\begin{align*}
\braket{d(k),\cG_{\mrm{A}}\eta(\chi)}_{\bbC^{md}}=&
\braket{(z-\hat{\Delta}^*)^{-1}[\cG_{\mrm{A}}d(k)]_m,
c}_{\bbC^d}
\\
=&-\braket{\hat{r}(z)(\cG^{\min}_{\mrm{A}})^{-1}
[\cG_{\mrm{A}}d(k)]_m,c}_{\bbC^d}\,.
\end{align*}
Putting $f^\#:=(L-z)^{-1}f\in\fH_{m+2}$
and applying Lemma~\ref{lem:BVSBmin-1}, \ie
\begin{align*}
d(k)=&\eta(\chi^\prime)+k^{\prime\prime}_\bot\,,
\quad k^{\prime\prime}_\bot\in\cH^\bot_{\mrm{A}}\,,
\quad \chi^\prime:=(\cG^{\min}_{\mrm{A}})^{-1}
[\cG_{\mrm{A}}d(k)]_m
\\
\Rightarrow&[\cG_{\mrm{A}}d(k)]_m=\cG^{\min}_{\mrm{A}}\chi^\prime\,,
\end{align*}
we deduce that
\begin{align*}
\ol{\ran}(\Bmin-z)=&
\{(L_m-z)f^\#+h_m(\chi)+k_\bot\vrt f^\#\in\fH_{m+2}\,;\,
\chi\in\bbC^d\,;
\\
&k_\bot\in\cH^\bot_{\mrm{A}}\,;\,
\braket{\vp,f^\#}=\hat{r}(z)\chi \}
=\ran(\Bmin-z)
\end{align*}
for $z\in\res L\mcap\res\hat{\Delta}$.
We remark that the functional
\[
\Phi(\cdot):=(\cG^{\min}_{\mrm{A}})^{-1}
[\cG_{\mrm{A}}d(\cdot)]_m\co\fK_{\mrm{A}}\lto\bbC^d
\]
is surjective, and that therefore $\chi^\prime=\Phi(k)$
ranges over all $\bbC^d$ whenever $k$ ranges over all
$\fK_{\mrm{A}}$.
This accomplishes the proof of the lemma.
\end{proof}
Because $\fN_z(\Bmax)\subseteq\cH^{\min}_{\mrm{A}}$,
$z\in\bbC$, and because
$\bbC\setm\bbR\subseteq\res L\mcap\res\hat{\Delta}$,
it follows that
\[
\mathring{\cH}_s:=\ol{\spn}\{\fN_z(\Bmax)\vrt z\in\bbC\setm\bbR\}
\subseteq\cH_s
\subseteq\cH^{\min}_{\mrm{A}}\,.
\]
By the proof of Lemma~\ref{lem:regBmin},
the orthogonal complement $\mathring{\cH}^\bot_s$
in $\cH_{\mrm{A}}$ of $\mathring{\cH}_s$ is given by
\[
\mathring{\cH}^\bot_s=
\bigcap_{z\in\bbC\setm\bbR}\ran(\Bmin-z)=
X[\dsum]\cH^{\bot}_{\mrm{A}}
\]
where the subset $X\subseteq\cH^{\min}_{\mrm{A}}$
is defined by
\[
X:=\{f+h_m(\chi)\in\fH_m\dsum\fK^{\min}_{\mrm{A}}\vrt
(\forall z\in\bbC\setm\bbR)\,
\braket{\vp,(L-z)^{-1}f}=
\hat{r}(z)\chi \}
\]
and $[\dsum]$ indicates the direct sum which is
orthogonal with respect to the $\cH_{\mrm{A}}$-metric
$[\cdot,\cdot]_{\mrm{A}}$. If $X=\{0\}$, \ie if
\eqref{eq:A3} has the only solutions $f=0$, $\chi=0$,
then $\mathring{\cH}^\bot_s=\cH^\bot_{\mrm{A}}$ implies
$\mathring{\cH}_s=\cH_s=\cH^{\min}_{\mrm{A}}$.
\end{proof}
Assuming the hypotheses in Theorem~\ref{thm:BVSBmin},
an extension
$B_\Theta\in\mrm{Ext}(\Bmin)$ parametrized by
a linear relation $\Theta$ in $\bbC^d$ is defined by
\[
B_\Theta:=\{\whf\in\Bmax\vrt\Gamma^\prime\whf\in\Theta\}\,.
\]
The Krein-Naimark
resolvent formula for $B_\Theta$ is given by
(\cf \cite[Theorem~4.12]{Derkach17})
\[
(B_\Theta-z)^{-1}=(B_0-z)^{-1}+
\gamma_{\Gamma^\prime}(z)(\Theta-M_{\Gamma^\prime}(z))^{-1}
\gamma_{\Gamma^\prime}(\ol{z})^*\,,
\quad
z\in\res B_0\mcap\res B_\Theta
\]
with $\gamma_{\Gamma^\prime}(\ol{z})^*=
\Gamma^\prime_1(B_0-z)^{-1}$. The self-adjoint
extension $B_0:=\ker\Gamma^\prime_0$ corresponds to
the self-adjoint linear relation $\Theta=\{0\}\times\bbC^d$.
The resolvent of $B_0$ is presented below.
\begin{prop}\label{prop:resB0}
Assuming the hypotheses in Theorem~\ref{thm:BVSBmin} we have
\[
(B_0-z)^{-1}(f+k)=(L_m-z)^{-1}f+
h_m((\hat{\Delta}-z)^{-1}\Phi(k))
\]
for $f\in\fH_m$, $k\in\fK_{\mrm{A}}$, and
$z\in\res B_0=\res L\mcap\res\hat{\Delta}$.
\end{prop}
\begin{proof}
By applying Lemma~\ref{lem:BVSBmin-1}
\[
B_0=\{(f^\#+h_m(\chi),L_mf^\#+h_m(\hat{\Delta}\chi)+
k_\bot)\vrt f^\#\in\fH_{m+2}\,;\,
\chi\in\bbC^d\,;\,k_\bot\in\cH^\bot_{\mrm{A}} \}\,.
\]
Thus the eigenspace
\[
\fN_z(B_0)=\fN_z(L_m)\dsum h_m(\fN_z(\hat{\Delta}))\,,
\quad z\in\bbC\,.
\]
From here we see that the point spectrum
\[
\sigma_p(B_0)=\sigma_p(L)\mcup\sigma_p(\hat{\Delta})\,.
\]
Then for $z\notin\sigma_p(B_0)$, the operator
\begin{align*}
(B_0-z)^{-1}=&\{(f+h_m(\chi)+k_\bot,(L_m-z)^{-1}f+
h_m((\hat{\Delta}-z)^{-1}\chi))\vrt
\\
&f\in\ran(L_m-z)\,;\,\chi\in\bbC^d \}
\end{align*}
and it therefore follows that
$\res B_0=\res L\mcap\res\hat{\Delta}$.
Putting $k:=h_m(\chi)+k_\bot$ we have that
$\chi=\Phi(k)$,
and this leads to
the resolvent formula as stated.
\end{proof}
In view of Proposition~\ref{prop:resB0},
the compressed resolvent
$P_{\fH_m}(B_\Theta-z)^{-1}\vrt_{\fH_m}$ is given
for $z\in\res B_0\mcap\res B_\Theta$ by
the right hand side of \eqref{eq:comp}, but where
now $M_\Gamma$ is replaced by $M_{\Gamma^\prime}$.


\end{document}